\documentclass[pra,twocolumn]{revtex4}%
\usepackage{amsfonts}
\usepackage{amsmath}
\usepackage{amssymb}
\usepackage{graphicx}%
\newtheorem{theorem}{Theorem}

\newtheorem{proposition}[theorem]{Proposition}

\newenvironment{proof}[1][Proof]{\noindent\textbf{#1.} }{\ \rule{0.5em}{0.5em}}
\begin{document}
\title{A note on observables for counting trails and paths in graphs}
\author{Fotini Markopoulou}
\affiliation{Perimeter Institute for Theoretical Physics, Waterloo N2L 2Y5, ON Canada}
\author{Simone Severini}
\affiliation{Institute for Quantum Computing and Department of Combinatorics \&
Optimization University of Waterloo, Waterloo N2L 3G1, ON Canada}
\date{24 June 2008}

\begin{abstract}
We point out that the total number of trails and the total number of paths of
given length, between two vertices of a simple undirected graph, are obtained
as expectation values of specifically engineered quantum mechanical
observables. Such observables are contextual with some background independent
theories of gravity and emergent geometry. Thus, we point out yet another
situation in which the mathematical formalism of a physical theory has some
computational aspects involving intractable problems.

\end{abstract}
\maketitle

\section{Introduction}

When the length is part of the input, counting trails and paths in graphs is
usually an expensive computational task. For example, counting the number of
Eulerian trails and Hamiltonian cycles are $\#P$ problems (see \cite{gw} and
\cite{v}, respectively).

In this note, we point out that the number of trails and the number of paths
having generic given length can be obtained as expectation values of
specifically engineered quantum mechanical observables. The observables arise
from an operational construction used to associate energy to graphs in certain
background indepedent models of gravity (see \cite{kms, kms1}). The
Hamiltonian of the system depends only on minimal information encoded in
graphs, like, for example, the degrees sequence and the length of trails. The
induced dynamics tends to maximize the number of trails of certain preferred
lengths during the time evolution of the system. It has been argued that such
models can exhibit a phase which describes a emergent geometry. We are
interested in remarking some computational aspects of a physical model and it
is out of our scope to propose any algorithm for $\#P$ problems.

The remainder of this article is structured as follows: in Section II we give
the necessary definitions. Section III considers trails; Section IV, paths.
Section V concludes the paper with a brief discussion.

\section{Counting trails and paths}

Let $G=(V,E)$ be a simple undirected graph, where $V(G)=\{1,...,n\}$. A
\emph{walk} of \emph{length} $l$ in $G$ is a non-empty sequence of vertices
$v_{1}v_{2}\cdots v_{l}v_{l+1}$, such that $\{v_{i},v_{i+1}\}\in E(G)$, for
every $i<l$. The vertices $v_{1}$ and $v_{l+1}$ are said to be the
\emph{initial} and \emph{final} vertex of the walk, respectively. If
$v_{i}=v_{l+1}$ then the walk is said to be \emph{closed}. If $v_{i}\neq
v_{l+1}$ then the walk is said to be \emph{open}. When this is the case, we
chose to omit the adjective. A \emph{trail} is a walk in which no edges are
repeated. A \emph{Eulerian trail} is a trail of including all edges. A graph
with a Euler trail is said to be \emph{Eulerian}. A \emph{path} is a trail in
which no vertices are repeated. A \emph{cycle} is a closed path. A
\emph{Hamiltonian cycle} is a cycle of length $n$, that is, a cycle including
all vertices of $G$. A graph with a Hamiltonian cycle is said to be
\emph{Hamiltonian}. See the book by Diestel \cite{die}, for a reference on the
concepts and terminology of graph theory.

The \emph{adjacency matrix} of $G$ is a binary $n\times n$ matrix, denoted by
$A(G)$, with $A(G)_{u,v}=1$ if $\{u,v\}\in E(G)$ and $A(G)_{u,v}=0$,
otherwise. Let $w(G,l,u,v)$ be the number of walks of length $l$ in $G$, with
initial vertex $u$ and final vertex $v$. It is well-known that
$w(G,l;u,v)=A(G)_{u,v}^{l}$, for all $u,v\in V(G)$, even if $u=v$. During our
discussion, it is useful to define a \emph{formal }adjacency matrix
$\widetilde{A}(G)$, by replacing each $A(G)_{u,v}=1$ with an independent
variable $e_{u,v}$, where $[e_{u,v},e_{w,z}]=0$, for all $u,v,w,z\in V(G)$.
For instance, for the $4$-cycle $C_{4}$,%
\[
A(C_{4})=\left(
\begin{array}
[c]{cccc}%
0 & 1 & 1 & 0\\
1 & 0 & 0 & 1\\
1 & 0 & 0 & 1\\
0 & 1 & 1 & 0
\end{array}
\right)
\]
and%
\[
\widetilde{A}(C_{4})=\left(
\begin{array}
[c]{cccc}%
0 & e_{12} & e_{13} & 0\\
e_{21} & 0 & 0 & e_{24}\\
e_{31} & 0 & 0 & e_{34}\\
0 & e_{42} & e_{43} & 0
\end{array}
\right)  .
\]
The walks of lengths $3$ between vertex $1$ and vertex $2$ are given by
$\widetilde{A}(C_{4})_{1,2}^{3}=e_{1,3}^{2}e_{1,2}+e_{1,2}e_{2,4}^{2}%
+e_{1,3}e_{3,4}e_{2,4}+e_{1,2}^{3}$ and then $w(C_{4},3;1,2)=4$. We denote by
$t(G,l;u,v)$ and $p(G,l;u,v)$ respectively the number of trails and paths of
length $l$ in $G$, with initial vertex $u$ and final vertex $v$.

\subsection{Trails}

Let $\mathcal{H}_{2}\cong\mathbb{C}^{2}$ be a two-dimensional Hilbert space
with the orthonormal basis $\{|0\rangle,|1\rangle\}$. Let $a$ and $a^{\dagger
}$ be operators obeying the relation
\[
\{a,a^{\dagger}\}\equiv aa^{\dagger}+a^{\dagger}a=I.
\]
These operators act on $\{|0\rangle,|1\rangle\}$ as follows: $a|0\rangle
=0|0\rangle$, $a^{\dagger}|0\rangle=|1\rangle$, $a|1\rangle=|0\rangle$ and
$a^{\dagger}|1\rangle=0|1\rangle$. The Hermitian combination $N=a^{\dagger}a$
is the so-called \emph{particle number operator} and it has the property that
\[
N|k\rangle=a^{\dagger}a|k\rangle=k|k\rangle,
\]
for $k=0,1$. Each $l$-th ($l\geq2$) normally ordered power of $N$ gives
$\colon N^{l}\colon|k\rangle=0|k\rangle$. For instance,%
\[
\colon N^{2}\colon|k\rangle=a^{\dagger}a^{\dagger}aa|k\rangle=0|k\rangle.
\]
This is equivalent to say $\colon N^{l}\colon$ only has zero eigenvalues. Let
us now define the space
\[
\mathcal{H}_{V^{2}}:=\bigotimes_{u,v\in V(G)}\left(  \mathcal{H}_{2}\right)
_{u,v},
\]
where $\dim\mathcal{H}_{V^{2}}=2^{\binom{n}{2}}$. Each pair $\{u,v\}$ is
associated to a space
\[
\left(  \mathcal{H}_{2}\right)  _{u,v}\equiv\text{ span}\{|0_{u,v}%
\rangle,|1_{u,v}\rangle\}.
\]
All vectors
\[
|x\rangle=\bigotimes_{u,v\in V(G)}|x_{u,v}\rangle,
\]
being $x_{u,v}\in\{0,1\}$ and $x\in\{0,1\}^{\binom{n}{2}}$ (\emph{i.e.}, the
set of binary strings of length $n(n-1)/2$) form an orthonormal basis of
$\mathcal{H}_{V^{2}}$. The state in $\mathcal{H}_{V^{2}}$ corresponding to the
graph $G$ is the basis state $|\psi_{G}\rangle$, in which $|x_{u,v}%
\rangle\equiv|1\rangle$ if $\{u,v\}\in E(G)$ and $|0\rangle$, otherwise. The
state $|\psi_{G}\rangle$ needs $n(n-1)/2$ qubits to be encoded. Operators
acting on the space $\mathcal{H}_{V^{2}}$ can be defined by making use of the
operators $a$ and $a^{\dagger}$ acting on the individual spaces $(\mathcal{H}%
_{2})_{u,v}$. In particular, it is possible to define number operators acting
on each copy of $(\mathcal{H}_{2})_{u,v}$. These operators are denoted by
$N_{u,v}$ and are defined by $N_{u,v}|x\rangle=x_{u,v}|x\rangle$, for
$x\in\{0,1\}^{\binom{n}{2}}$. Each operator $N_{u,v}$ returns the occupation
number $x_{u,v}$. More explicitly, $N_{u,v}$ is defined by
\[
N_{u,v}|x\rangle=1\otimes1\otimes\cdots\otimes1\otimes a^{\dagger}%
a\otimes1\otimes\cdots\otimes1|x\rangle=x_{u,v}|x\rangle.
\]
Since the operators $N_{u,v}$ act on different subsystems of $\mathcal{H}%
_{V^{2}}$, they all commute with each other.

Now, let us define a matrix $\mathbf{N}$ with elements
\[
\mathbf{N}_{u,v}=\left\{
\begin{array}
[c]{ll}%
N_{u,v}, & \text{{if}}\,u\neq v\text{;}\\
0, & \text{{otherwise}.}%
\end{array}
\right.
\]
Note that the matrix $\mathbf{N}$ is not an operator on $\mathcal{H}_{V^{2}}$
in the usual sense; it is the elements of $\mathbf{N}$ that act naturally on
$\mathcal{H}_{V^{2}}$. Thus, the action of $\mathbf{N}$ on a state should be
understood as occurring entry-wise. The evaluation of the matrix $\mathbf{N}$,
using the state $|\psi_{G}\rangle$, gives the same entries of the adjacency
matrix, that is,
\[
\langle\psi_{G}|\mathbf{N}_{u,v}|\psi_{G}\rangle:=\langle\psi_{G}|N_{u,v}%
|\psi_{G}\rangle=A(G)_{u,v}.
\]

Powers of $\mathbf{N}$ can be defined recursively as $\mathbf{N}%
^{l}=\mathbf{N}\mathbf{N}^{l-1}$. The entries of $\mathbf{N}^{l}$ are sums of
sequences of number operators acting on different subspaces. By considering
the expectation values of the normally ordered entries, we can state the
following proposition: (Recall that a trail is a walk in which no edges are repeated.)

\begin{proposition}
\label{pro}Given a graph $G$, the total number of trails of length $l$ in $G$,
with initial vertex $u$ and final vertex $v$ is
\begin{equation}
t(G,l;u,v)=\langle\psi_{G}|\colon\mathbf{N}_{u,v}^{l}\colon|\psi_{G}\rangle.
\label{Pnonover}%
\end{equation}

\end{proposition}

\begin{proof}
(Sketch) Keeping in mind the formal adjacency matrix $\widetilde{A}(G)$, we
need to point out the following two facts only:

\begin{itemize}
\item If a product of number operators in one of the terms $\mathbf{N}%
_{u,v}^{l}$ contains a member acting on a pair of nonadjacent vertices then
the operator yields zero and that term does not contribute to the expectation value.

\item When a term in $\mathbf{N}_{u,v}^{l}$ contains more than one copy of a
number operator acting on a particular edge, that term also does not
contribute to the expectation value, because $\colon N^{l}\colon
|k\rangle=0|k\rangle$.
\end{itemize}

It follows that the only combinations of number operators that do not give a
vanishing contribution to $t(G,l,u,v)$ correspond to trails with initial
vertex $u$ and final vertex $v$. Thus, Eq. (\ref{Pnonover}) gives this number
of trails as an expectation value of a quantum mechanical operator.
\end{proof}

\bigskip

Notice that $\colon\mathbf{N}_{u,u}^{m}\colon$ counts the number of Eulerian
trails in $G$, if $|E(G)|=m$.

It is noteworthy to remark that the logarithm of the dimension of the space
$\mathcal{H}_{V^{2}}$ is polynomial in the number of vertices. Also, notice
that the operators $\mathbf{N}_{u,v}^{l}$ are independent of $G$. For the
purpose of counting trails in a specific graph, a similar procedure may be
applied taking a Hilbert space
\[
\mathcal{H}_{E}:=\bigotimes_{\{u,v\}\in E(G)}(\mathcal{H}_{2})_{u,v},
\]
where $\dim\mathcal{H}_{E}=2^{|E(G)|}$. In this case, the matrix of operators
$\mathbf{N}$ needs to be slightly modified so that some of its entries are
zero from the beginning, rather than number operators acting on empty states.

\bigskip

Proposition \ref{pro} is based on the equation $\colon N^{l}\colon
|k\rangle=0|k\rangle$. It may be worth remarking that there is way of counting
trails without making use of normal ordering. This can be done by defining the
matrices
\[
\mathbf{D}_{u,v}=\left\{
\begin{array}
[c]{ll}%
a_{u,v}, & \text{if }u\neq v;\\
0, & \text{otherwise.}%
\end{array}
\right.
\]
It is in the same spirit of matrix $\mathbf{N}_{u,v}$, but the entries are
annihilation operators, $a_{u,v}$, rather than number operators $N_{u,v}$.
This matrix of operators is not Hermitian, \emph{i.e.}, $\mathbf{D}^{\dagger
}\neq\mathbf{D}$. Here the Hermitian conjugate is defined entry-wise by
$(\mathbf{D}^{\dagger})_{u,v}=\left(  \mathbf{D}_{u,v}\right)  ^{\dagger}$. It
can be shown that%
\begin{align*}
t(G,l;u,v)  &  =\langle\psi_{G}|\colon\mathbf{N}_{u,v}^{l}\colon|\psi
_{G}\rangle\\
&  =\langle\psi_{G}|\,\left(  \mathbf{D}_{u,v}^{l}\right)  ^{\dagger
}\mathbf{D}_{u,v}^{l}\,|\psi_{G}\rangle.
\end{align*}

\subsection{Example}

We consider the $4$-cycle $C_{4}$: the set of vertices is $V(C_{4}%
)=\{1,2,3,4\}$; the set of edges is $E(C_{4}%
)=\{\{1,2\},\{1,3\},\{2,4\},\{3,4\}\}$. Since $|V(C_{4})|=4$, we associated to
this graph an Hilbert space $\mathcal{H}_{V^{2}}$ of dimension $2^{6}=64$.
Specifically,%
\[
\mathcal{H}_{V^{2}}=\left(  \mathcal{H}_{2}\right)  _{1,2}\otimes\left(
\mathcal{H}_{2}\right)  _{1,3}\otimes\cdots\otimes\left(  \mathcal{H}%
_{2}\right)  _{3,4}.
\]
The state associated to $C_{4}$ is
\[
|\psi_{G}\rangle=|110011\rangle.
\]
Regarding the operators $N_{u,v}$, we have, for instance%
\[
N_{1,2}|110011\rangle=1|110011\rangle
\]
and
\[
N_{1,4}|110011\rangle=0|110011\rangle.
\]
In fact,
\[
\langle110011|N_{1,2}|110011\rangle=A(C_{4})_{1,2}=1
\]
and
\[
\langle110011|N_{1,4}|110011\rangle=A(C_{4})_{2,4}=0
\]
Observe that $\widetilde{A}(C_{4})_{1,2}^{3}=e_{1,3}^{2}e_{1,2}+e_{1,2}%
e_{2,4}^{2}+e_{1,3}e_{3,4}e_{2,4}+e_{1,2}^{3}$. For the vertices $1$ and $2$,
we have,%
\begin{align*}
\langle\psi_{G}|\colon\mathbf{N}_{1,2}^{3}\colon|\psi_{G}\rangle &
=\langle\psi_{G}|\colon N_{1,3}^{2}N_{1,2}\colon|\psi_{G}\rangle\\
&  +\langle\psi_{G}|\colon N_{1,2}N_{2,4}^{2}\colon|\psi_{G}\rangle\\
&  +\langle\psi_{G}|\colon N_{1,3}N_{3,4}N_{2,4}\colon|\psi_{G}\rangle\\
&  +\langle\psi_{G}|\colon N_{1,2}^{3}\colon|\psi_{G}\rangle\\
&  =\langle\psi_{G}|\colon N_{1,3}N_{3,4}N_{2,4}\colon|\psi_{G}\rangle\\
&  =1
\end{align*}
Indeed, $t(C_{4},3;1,2)=p(C_{4},3;1,2)=1$, something that $\colon
\mathbf{N}_{1,2}^{3}\colon$ is able to detect even if $w(C_{4},3;1,2)=4$.

\subsection{Paths}

In this section, our working space is $\mathcal{H}_{V}:=\bigotimes_{v\in
V(G)}(\mathcal{H}_{2})_{v}$, where $\dim\mathcal{H}_{V}=2^{n}$, given that
$|V(G)|=n$. This is the space usually assigned to networks of spin $1/2$
quantum mechanical particles. The space $(\mathcal{H}_{2})_{v}$ is associated
to the vertex $v$ and $\left(  \mathcal{H}_{2}\right)  _{v}\equiv$
span$\{|0_{v}\rangle,|1_{v}\rangle\}$. All vectors $|x\rangle=\bigotimes_{v\in
V(G)}|x_{v}\rangle$, being $x_{v}\in\{0,1\}$ and $x\in\{0,1\}^{n}$, form an
orthonormal basis of $\mathcal{H}_{V}$. The state in $\mathcal{H}_{V}$
corresponding to the graph $G$ is the basis state $|11...1\rangle
=\bigotimes_{n}|1\rangle$. As we have done in the previous subsection, we can
define number operators acting on each $(\mathcal{H}_{2})_{v}$. These
operators are denoted by $N_{v}$.

Let $\mathbf{M}$ be a matrix of operators with entries defined as follows:
\[
\mathbf{M}_{u,v}=\left\{
\begin{array}
[c]{ll}%
N_{\mathrm{max}(u,v)}, & \text{if}\,\{u,v\}\in E(G)\text{ and }v>u\text{;}\\
N_{\mathrm{min}(u,v)}, & \text{if }\{u,v\}\in E(G)\text{ and }u>v\text{;}\\
0, & \text{otherwise.}%
\end{array}
\right.
\]
As in the previous subsection, the elements of $\mathbf{M}$ are still number
operators, but this time acting on $\mathcal{H}_{V}$. However, unlike the
adjacency matrix, $\mathbf{M}$ is not symmetric. Powers of $\mathbf{M}$ are
defined via matrix multiplication, for example, by the recursion
$\mathbf{M}^{l}=\mathbf{M}\mathbf{M}^{l-1}$. Recall that a path is a trail in
which no vertices are repeated.

\begin{proposition}
Given a graph $G$, the total number of paths of length $l$ in $G$, with
initial vertex $u$ and final vertex $v$ is
\[
p(G,l;u,v)=\langle11...1|\colon\mathbf{M}_{u,v}^{l}\colon|11...1\rangle.
\]

\end{proposition}

\begin{proof}
(Sketch) The operators $\mathbf{M}_{u,v}^{l}$ are sums of products of number
operators. Since the entries in $\mathbf{M}$ are nonzero only in the positions
where the adjacency matrix is nonzero, each of the terms in $\mathbf{M}%
_{u,v}^{l}$ also correspond to walks. Due to the normal ordering convention, a
term in which the same vertex appears more than once does not contribute to
the expectation value. Thus, $\langle\varphi_{G}|\colon\mathbf{M}_{u,v}%
^{l}\colon|\varphi_{G}\rangle$ has the desired interpretation.
\end{proof}

\bigskip

A similar construction could be made up with a \textquotedblleft symmetric
version\textquotedblright\ of $\mathbf{M}$.

With respect to the Hilbert space $\mathcal{H}_{V}$, we define a matrix of
operators $\mathbf{F}$ with entries
\[
\mathbf{F}_{u,v}=\left\{
\begin{array}
[c]{ll}%
a_{\mathrm{max}(u,v)} & \text{if}\,\{u,v\}\in E(G)\text{ and }v>u\text{;}\\
a_{\mathrm{min}(u,v)} & \text{if }\{u,v\}\in E(G)\text{ and }u>v\text{;}\\
0 & \text{otherwise.}%
\end{array}
\right.
\]
Notice that $\mathbf{F}_{u,v}$ depends on $E(G)$ and thus on the graph $G$.
Then, the transition amplitude
\[
p(G,l;u,u)=\langle00...0|\mathbf{F}_{u,u}^{l}|11...1\rangle,
\]
is the number of cycles of length $l$ containing the vertex $u$. Taking $l=n$,
we have $p(G,n;u,u)>0$ if and only if the graph $G$ used to construct
$\mathbf{F}_{u,u}$ is Hamiltonian.

\section{Conclusion}

For the purpose of counting trails and paths, we have described quantum
observables in Hilbert spaces whose logarithm of the dimension is a polynomial
in the number of vertices. The numbers of trails and paths can be obtained as
expectation values of these observables. The states involved are the pure
states $|11...1\rangle$ and $|\psi_{G}\rangle$. These states can be prepared
efficiently. While it is clear that the observables raise no issues about
uncertainty, it is most likely that these are not efficiently implementable in
a quantum computer (\emph{e.g.}, by phase estimation). The reason behind this
thought is based on the fact that our observables are used to approach $\#P$
problems. Despite this, it may still be instructive to describe their form for
special classes of graphs, to determine their complexity, and to describe the
physics required for the implementation.

It is valuable to point out that the literature contains so far a number of
examples of quantum observables for solving computational task, which are
either not known to be efficiently implementable or implementable with poor
accuracy. Among these, observables for the graph isomorphism problem \cite{e}
and for the permanent \cite{t} (see also \cite{v1}). It has also be shown that
some \emph{mathematically well-defined }observables allow to solve the halting
problem, in contradiction with the Church-Turing thesis \cite{n}. Recall that
the Church-Turing thesis asserts that every function which can be computed by
what we would naturally regard as an algorithm is a computable function, and
\emph{viz.}

In the present notes, we have highlighted yet another situation in which the
mathematical formalism of a physical theory has some computational aspects
involving intractable problems. On the other hand, the computational
complexity side may suggest some potential limitations on the physical picture
(see \cite{aa}, for a detailed survey on these ideas) -- that is, the physical
picture is simply not plausible. -- With the same perspective of \cite{l}, it
is legitimate to study the properties of background independent models of
gravity and emergent geometry as computational devices (as cellular automata,
for instance). Indeed, in the model studied in \cite{kms}, the Hamiltonian of
the system keeps track of trails of given lengths. Still this is inline with
natural phenomena of self-organization at the microscopic scale. For example,
the folded 3-dimensional conformation of a protein is believed to be its
lowest free energy state. The 3-dimensional models describing the folding
process of a protein as free energy minimization problems are $NP$-hard (see,
\emph{e.g.}, \cite{c}). Finally, since the entries of the observables
described in this paper can be seen as Grassmann numbers (because
corresponding to fermionic operators), the observables may have some analogy
to certain matrix model (see, \emph{e.g.}, \cite{s}), where the partition
function is given by their weighted sums.

\bigskip

\noindent\emph{Acknowledgments. }We thank Tomasz Konopka for important
discussion, Andrew Childs and Peter H\o yer for their comments. This research
was conducted while the authors were at the Massachusetts Institute of
Technology. Financial support by fqxi is gratefully acknowledged.

\end{document}